\documentclass[a4paper,11pt]{article}

\usepackage[T1]{fontenc} 
\usepackage[utf8]{inputenc}
\usepackage[ngerman,american]{babel} 

\usepackage{amsmath}
\usepackage{amsfonts}
\usepackage{amssymb}
\usepackage{graphicx}
\usepackage{color}
\usepackage{tikz} 
\usepackage{hhline} 

\newcommand{\cellwidth}{0.8cm}
\newcommand{\celldist}{0.8cm}

\newcommand{\resolvedcellwidth}{0.6cm}
\newcommand{\resolvedcelldist}{0.6cm}

\newcommand{\ranktwocolor}{gray!20}
\newcommand{\rankonecolor}{gray!20}
\newcommand{\bothoddcolor}{gray!50}
\newcommand{\byhandcolor}{red}
\newcommand{\bycomputercolor}{blue}

\usetikzlibrary[positioning]
\usetikzlibrary{patterns}
\usetikzlibrary{fadings}

\usepackage{algorithm}
\usepackage{algorithmic}

\newcommand{\ENSUREGAP}{\vspace{0.2cm}}

\tikzstyle{top}=[draw, rectangle,  minimum height=\cellwidth, minimum width=\cellwidth, fill=gray!10, , text height = 0.25cm]
\tikzstyle{left}=[draw, rectangle, minimum height=\cellwidth, minimum width=\cellwidth, fill=gray!10, ]
\tikzstyle{cell}=[rectangle,draw, minimum height=\cellwidth,minimum width=\cellwidth,text centered,text width=0.5cm]
\tikzstyle{invisible}=[rectangle, minimum height=\cellwidth,minimum width=\cellwidth,text centered,text width=0.5cm]
\tikzstyle{resolvedcell}=[rectangle,draw, minimum height=\resolvedcellwidth,minimum width=\resolvedcellwidth,text centered,text width=0.03cm]

\newcommand{\glabels}[1]
{
    \begin{scope}
	    \foreach \x in {1,...,#1} 
	    {
	    	\node[top] at (\x,0) {$h_\x$};
	    }
    \end{scope}
}
\newcommand{\hlabels}[1]
{
    \begin{scope}
	    \foreach \x in {1,...,#1} 
	    {
	    	\node[top] at (0,-\x) {$g_\x$};
	    }
    \end{scope}
}
\newcommand{\recblocks}[2]
{
    \begin{scope}
    \foreach \x in {1,...,#1} 
	{
			\foreach \y in {1,...,#2} 
			{
				\node[cell]  at (\x,-\y) {};
			}
	}
    \end{scope}
}
\newcommand{\resolvedblocks}[2]
{

    \begin{scope}
    \foreach \x in {1,...,#1} 
	{
			\foreach \y in {1,...,#2} 
			{
				\node[resolvedcell]  at (\x,-\y) {};
			}
	}
    \end{scope}
    	\foreach \x in {1,...,#1} 
    	{
    		\node  at (\x,0) {$\x$};
    		\node[resolvedcell, fill=\rankonecolor]  at (\x,-1) {};
    		\node[resolvedcell, fill=\ranktwocolor]  at (\x,-2) {};
    		    		    			
    	}
    	\foreach \y in {1,...,#2} 
    	{
    		\node  at (0,-\y) {$\y$};	
    		\node[resolvedcell, fill=\rankonecolor]  at (1,-\y) {};
    		\node[resolvedcell, fill=\ranktwocolor]  at (2,-\y) {};
    	}
    	
}

\usepackage[lofdepth,lotdepth]{subfig}

\usepackage{fullpage}
\usepackage{amsthm}
\theoremstyle{plain}
\newtheorem{theorem}{Theorem}
\newtheorem{lemma}[theorem]{Lemma}

\newtheorem{corollary}[theorem]{Corollary}
\theoremstyle{definition}
\newtheorem{definition}[theorem]{Definition}

\newtheorem{conjecture}{Conjecture}

\usepackage[colorlinks]{hyperref}

\definecolor{lightblue}{rgb}{0.5,0.5,1.0}
\definecolor{darkred}{rgb}{0.5,0,0}
\definecolor{darkgreen}{rgb}{0,0.5,0}
\definecolor{darkblue}{rgb}{0,0,0.5}

 \hypersetup{colorlinks
 ,linkcolor=darkred
 ,filecolor=darkgreen
 ,urlcolor=darkred
 ,citecolor=darkblue}

\DeclareMathOperator{\supp}{supp}
\DeclareMathOperator{\cyc}{cyc}

\DeclareMathOperator{\notorsion}{tor-free}
\DeclareMathOperator{\assocgroup}{G_{ass}}
\DeclareMathOperator{\assocgraph}{K_{ass}}
\DeclareMathOperator{\length}{length}

\DeclareMathOperator{\match}{match}
\DeclareMathOperator{\lex}{lex}

\newcommand{\BigO}{O}

\title{On Zero Divisors with Small Support in Group Rings of Torsion-Free Groups}
\author{Pascal Schweitzer\thanks{This
work is supported 
by the National Research Fund of Luxembourg, and co-funded under the Marie Curie 
Actions of the European Commission (FP7-COFUND). } \\[2ex]
Research School of Computer Science\\
The Australian National University\\ 
Canberra, ACT 0200, Australia\\
{Pascal.Schweitzer@anu.edu.au} \\[1ex]
}

\begin{document}

\maketitle

\begin{abstract}
Kaplanski's Zero Divisor Conjecture envisions that for a torsion-free group~$G$ and an integral domain~$R$, the group ring $R[G]$ does not contain non-trivial zero divisors. We define the length of an element~$\alpha\in R[G]$ as the minimal non-negative integer~$k$ for which there are ring elements~$r_1,\ldots,r_k\in R$ and group elements~$g_1,\ldots,g_k\in G$ such that~$\alpha = r_1 g_1+\ldots+r_kg_k$. We investigate the conjecture when~$R$ is the field of rational numbers.
By a reduction to the finite field with two elements, we show that if~$\alpha \beta = 0$ for non-trivial elements in the group ring of a torsion-free group over the rationals, then the lengths of~$\alpha$ and~$\beta$ cannot be among certain combinations. More precisely, we show for various pairs of integers~$(i,j)$ that if one of the lengths is at most~$i$ then the other length must exceed~$j$.
Using combinatorial arguments we show this for the pairs~$(3,6)$ and~$(4,4)$. With a computer-assisted approach we strengthen this to show the statement holds for the pairs~$(3,16)$ and~$(4,7)$. As part of our method, we describe a combinatorial structure, which we call matched rectangles, and show that for these a canonical labeling can be computed in quadratic time. Each matched rectangle gives rise to a presentation of a group. These associated groups are universal in the sense that there is no counterexample to the conjecture among them if and only if the conjecture is true over the rationals.
\end{abstract}

\section{Introduction}
The study of group rings was initiated in~1837 by Hamilton in order to study first the complex numbers and later the quaternions (see~\cite{Milies}). 
Given a ring~$R$ and a group~$G$, the group ring~$R[G]$ of~$G$ over~$R$ is the ring whose elements are the linear combinations of elements in~$G$ with coefficients in~$R$. The multiplication in~$R[G]$ is the linear extension of the multiplication in~$G$.
Recall that if for $\alpha,\beta \in R[G]\setminus \{0\}$ we have $\alpha  \beta = 0$ then~$\alpha$ as well as~$\beta$ are called \emph{non-trivial zero divisors}.
Around 1940 Kaplanski asked whether for an integral domain~$R$ (i.e., a commutative ring without non-trivial zero divisors) the group ring~$R[G]$ of a torsion-free group~$G$ over~$R$ can have non-trivial zero divisors (see~\cite{Passman}).

\begin{conjecture}[Kaplanski's Zero Divisor Conjecture]
If~$G$ is a torsion-free group and~$R$ an integral domain, then the group ring $R[G]$ does not contain non-trivial zero divisors.
\end{conjecture}
Passman~\cite{Passman} noted 37 years later that ``with the possible exception of semi-simplicity questions, this is by
far the most challenging and least approachable problem in the whole field
of group rings of infinite groups''.
Over the years, the conjecture has emerged as a most elementary yet fundamental question in algebra. Progress concerning the conjecture has been made in various directions but the conjecture itself remains open. 

The conjecture has been shown for various types of torsion-free groups. First known for right orderable groups (see~\cite{MR1709960}), Lewin~\cite{Lewin} showed the conjecture holds for amalgamated free products when the group ring of the subgroup over which the amalgam is formed satisfies the Ore condition. Using this, Formanek~\cite{Formanek} showed the conjecture for supersolvable groups.  Subsequent extension by Brown~\cite{Brown} and Farkas and Snider~\cite{Farkas} led to polycyclic-by-finite groups. This was again extended by Snider~\cite{MR599330}. Finally Kropholler, Linnell, and Moody~\cite{MR964842} use~$K$-theory to resolve the conjecture for elementary amenable groups. In other directions Lewin and Lewin~\cite{MR0485972} showed the conjecture for one-relator groups, Lazard~\cite{MR0209286} treated congruence subgroups (see also~\cite{MR2279234}) and Delzant~\cite{MR1440952} certain hyperbolic groups. 

For the case where~$R$ is isomorphic to the integers or to the field~$\mathbb{F}_p$ that has~$p$ elements for some prime~$p$, Lichtman~\cite{MR835867} gave a reformulation of the conjecture as group theoretic problem dealing solely with subgroups of a free group.
Linnell~\cite{MR1709960}, also surveying the status of the conjecture, considered an analytic variant. Subsequently, Elek~\cite{GaborElek} showed this analytic variant to be equivalent to the original conjecture. Ivanov~\cite{Ivanov199913} has exhibited a connection between asphericity of specific topological spaces and the Zero Divisor Conjecture for which Leary~\cite{Leary2000362} gave a short variant of the proof. Along these lines recently, Kim showed asphericity of certain one-relator presentations~\cite{MR2391638}.

In this paper we pursue an algorithmic approach to the Zero Divisor Conjecture. To this end, we explain in the following section why it is reasonable to consider the problem over~$\mathbb{F}_2$, the field that contains exactly two elements. We describe how having zero divisors~$\alpha$ and~$\beta$ with~$\alpha \beta = 0$ in a group ring~$\mathbb{F}_2[G]$ gives rise to an associated torsion-free, finitely presented group, which also has zero divisors. 
For brevity, we call a torsion-free group, which has non-trivial zero divisors over~$\mathbb{F}_2$, a counterexample. The associated groups are universal among the counterexamples in the sense that every counterexample~$G$ contains a subgroup~$H'$, which is a counterexample and which is a factor group of a counterexample among the associated groups.
Similar considerations leading to a universal family of finitely presented groups have been performed independently by other researchers~\cite{dynkema,RomanMikhailov}.

As explained, it suffices for us to consider associated groups. There is a correspondence between these groups and a class of combinatorial objects, which we call matched rectangles. 

The length of an element~$\alpha\in R[G]$ is the minimal non-negative integer~$k$ for which there are ring elements~$r_1,\ldots,r_k\in R$ and group elements~$g_1,\ldots,g_k\in G$ such that~$\alpha = r_1 g_1+\ldots+r_kg_k$.
When the lengths of the zero divisors~$\alpha$ and~$\beta$ are bounded, there are only finitely many associated groups that arise from the construction mentioned above.
Using combinatorial arguments, we show for specific length bounds that none of these associated groups are counterexamples, and thus prove the following theorem.

\begin{theorem}\label{thm:by:hand}
Let~$G$ be a torsion-free group and let~$\alpha,\beta \in \mathbb{Q}[G]\setminus \{0\}$, with~$\length(\alpha) = n$ and~$\length(\beta) = m$, be non-zero elements of the group ring of~$G$ over~$\mathbb{Q}$.
If~$\alpha  \beta = 0$ then
\begin{enumerate}
\item~$n >2$\label{item1},
\item~$m >2$\label{item2},
\item~$n >3$ or~$m>6$\label{item3},
\item~$n >6$ or~$m>3$\label{item4}, and
\item~$n >4$ or~$m>4$\label{item5}.
\end{enumerate}
\end{theorem}

For greater length combinations, the arguments required to show that none of the associated groups are counterexamples become excessively tedious and amount to a large case distinction. This calls for a computer-assisted approach. Using the canonical construction path method by McKay~\cite{canonicalMcKay}, we design an algorithm that, for a fixed length combination~$(n,m)$, enumerates all minimal matched rectangles of dimensions at most~$(n,m)$, which correspond to counterexamples. 
We obtain the following strengthened variant of Theorem~\ref{thm:by:hand}.

\begin{theorem}[computer-assisted]\label{thm:by:computer}
Let~$G$ be a torsion-free group and let~$\alpha,\beta \in \mathbb{Q}[G]\setminus \{0\}$, with~$\length(\alpha) = n$ and~$\length(\beta) = m$, be non-zero elements of the group ring of~$G$ over~$\mathbb{Q}$.
If~$\alpha \beta = 0$ then
\begin{enumerate}
\item~$n >2$\label{item1b},
\item~$m >2$\label{item2b},
\item~$n >3$  or~$m>16$\label{item3b},
\item~$n >16$ or~$m>3$\label{item4b}, and
\item~$n >4$  or~$m>7$\label{item5b}.
\item~$n >7$  or~$m>4$\label{item6b}.
\end{enumerate}
\end{theorem}

Figure~\ref{fig:length:combinations} visualizes the length combinations excluded by Theorems~\ref{thm:by:hand} and~\ref{thm:by:computer}.

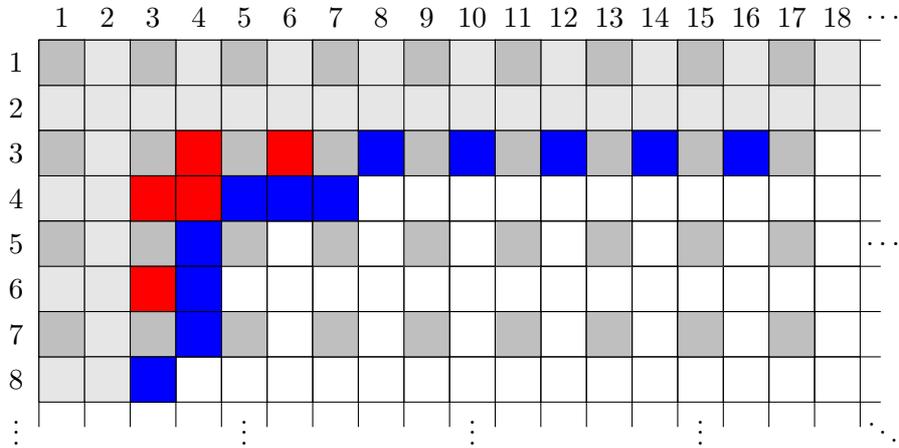
\begin{figure}[ht]
\centering
\begin{tikzpicture}[x=\resolvedcelldist, y=\resolvedcelldist, node distance=0 cm,outer sep = 0pt]

\resolvedblocks{18}{8}
\foreach \x/\y in {{1/1},{1/3},{1/5},{1/7},{3/1},{3/3},{3/5},{3/7},{5/1},{5/3},{5/5},{5/7},{7/1},{7/3},{7/5},{7/7},{9/1},{9/3},{9/5},{9/7},{11/1},%
{11/3},{11/5},{11/7},{13/1},{13/3},{13/5},{13/7},{15/1},{15/3},{15/5},{15/7},{17/1},{17/3},{17/5},{17/7}}
{
	\node[resolvedcell, fill=\bothoddcolor]  at (\x,-\y) {};
}
\foreach \x/\y in {{3/4},{4/3},{3/6},{6/3},{4/4}}
{
	\node[resolvedcell, fill=\byhandcolor]  at (\x,-\y) {};
}	
\foreach \x/\y in {{3/8},{8/3},{10/3},{12/3},{14/3},{16/3}, {4/5},{5/4},{4/6},{6/4},{4/7},{7/4}}
{
	\node[resolvedcell, fill=\bycomputercolor]  at (\x,-\y) {};
}	
	\node  at (19,0) {$\hdots$};
	\node  at (19,-5) {$\hdots$};
	\node  at (0,-9) {$\vdots$};
	\node  at (5,-9) {$\vdots$};
	\node  at (10,-9) {$\vdots$};
	\node  at (15,-9) {$\vdots$};
	\node  at (19,-9) {$\ddots$};
	\foreach \x in {1,...,19} 
	{
	    	\draw[-] (\x,-9) +(-0.5,0.5)  -- ++(-0.5,-0.05) {};
	}
	\foreach \y in {1,...,9} 
	{
	    	\draw[-] (19,-\y) +(-0.5,0.5)  -- ++(-0.05,0.5) {};
	}

\end{tikzpicture}
\caption{The figure depicts the various length combinations for which zero divisors cannot occur over~$\mathbb{Q}$. The gray shaded regions cannot occur due to the general statements for length at most~$2$ (Theorem~\ref{thm:length:2}) and due to the fact that one of the lengths must be even. Exclusion of the darker red and blue shaded regions is proven by combinatorial arguments (Theorem~\ref{thm:by:hand}) and with the aid of a computer (Theorem~\ref{thm:by:computer}), respectively.}\label{fig:length:combinations}
\end{figure}
\paragraph{Structure of the paper:} In Section~\ref{sec:matched:recs} we define matched rectangles and their associated groups and prove their universality with respect to the Zero Divisor Conjecture. In Section~\ref{sec:BS:groups} we briefly argue that all quotients of a solvable Baumslag-Solitar group satisfy the Zero Divisor Conjecture. Using this and other combinatorial arguments we then prove Theorem~\ref{thm:by:hand} in Section~\ref{sec:by:hand}. In Section~\ref{sec:canonical:labeling} we show how to label canonically matched rectangles without proper sub-rectangles in polynomial time and use this in Section~\ref{sec:by:computer} to prove Theorem~\ref{thm:by:computer} with the help of a computer. We conclude in Section~\ref{sec:conclusion}.

\section{Matched rectangles}\label{sec:matched:recs}

Let~$R$ be an integral domain and~$G$ a torsion-free group. It is well known that an element of length 2 cannot be a zero divisor in the group ring~$R[G]$. 
We start by showing this statement using a proof technique that captures some essential ideas of our approach.

\begin{theorem}\label{thm:length:2}
Let~$G$ be a torsion-free group and~$R$ be an integral domain. If~$\alpha \in R[G]\setminus \{0\}$ has length at most 2, then~$\alpha$ is not a zero divisor.
\end{theorem}
\begin{proof}
The statement is obvious if~$\alpha$ has length 1. Since every integral domain embeds into a field, we can assume w.l.o.g. that~$R$ is a field.
Suppose a zero divisor $\alpha$ has length 2 in a torsion-free group G.  By symmetry it suffices to show that~$\alpha$ is not a left zero divisor, i.e., we may assume for contradiction that there exists a $\beta \in R[G]\setminus \{0\}$ such that $\alpha \beta = 0$. By multiplying with a suitable group element and a suitable ring element from the left we can see that we can choose~$\alpha$ such that $\alpha = 1 + r g$ for some $1\neq g \in G$ and~$r\in R$. Similarly, by multiplication from the right, we can show that~$\beta$ can be chosen to have the form $\beta= 1 + \ell_2 h_2 + \ldots + \ell_m h_{m}$, where $m$ is the length of $\beta$, and all $h_i$ are distinct and different from~1. For notational simplicity we define $h_1 = 1$ and~$\ell_1 = 1$.

Since $\alpha  \beta = 1 + \ell_2 h_2 + \ldots + \ell_m h_{m} + g + r \ell_2 g h_2 + \ldots + r \ell_m g h_{m} = 0$, and since $g h_i \neq g h_j$ for $i\neq j$, there is a bijection~$\phi\colon  \{1,\ldots,m\}\rightarrow \{1,\ldots,m\}$ such that~$h_i = g h_{\phi(i)}$. 

We now argue that there is a~$t\in \{1,\ldots,m\}$ such that~$g^t = 1$. By induction on~$k$ we see that for all~$i\in\{1,\ldots,m\}$ and all~$k\in \mathbb{Z}$ we have~$h_i = g^k h_{\phi^k(i)}$. Since~$\phi$ is a permutation, there is a~$t>0$ such that~$\phi^t(1) = 1$ and thus~$1 = h_1 = g^t h_1  = g^t$.
This shows that G has torsion and yields a contradiction.
\end{proof}

A similar statement for group ring elements~$\alpha$ of length at most~$3$ is not known. Note that Kaplanski's Zero Divisor Conjecture is equivalent to the theorem being true for all lengths. 
Though it is not clear how to show a similar statement for length~$3$, for some rings we can show statements of the following form. If $\alpha \beta = 0$ then $\alpha$ must have length longer than $n\in \mathbb{N}$ or $\beta$ must have length longer than $m\in \mathbb{N}$. To prove a statement of this form, we first explain how to reduce the problem to a statement over the field~$\mathbb{F}_2$.

\begin{lemma}\label{lem:Q:to:FP}
Let~$G$ be a group. If~$\alpha,\beta \in \mathbb{Q}[G]\setminus\{0\}$ with~$\alpha  \beta = 0$ then for every prime number~$p$ the group ring~$\mathbb{F}_p [G]$ has elements~$\alpha',\beta' \in \mathbb{F}_p[G]\setminus\{0\}$ with~$\alpha'  \beta' = 0$ such that~$\length(\alpha') \leq \length(\alpha)$ and~$\length(\beta') \leq \length(\beta)$.
\end{lemma}

\begin{proof}

By multiplying $\alpha$ and $\beta$ with suitable rationals, we can achieve that both $\alpha$ and $\beta$ contain only integral coefficients and that they both contain a coefficient that is not divisible by~$p$. 

In this case, the canonical projections onto $\mathbb{F}_p[G]$ of both $\alpha$ and $\beta$ are non-trivial elements~$\alpha'$ and~$\beta'$ with~$\length(\alpha') \leq \length(\alpha)$ and~$\length(\beta') \leq \length(\beta)$. Moreover~$\alpha' \beta'= 0$ in~$\mathbb{F}_p[G]$.
\end{proof}

\begin{corollary}
Let~$G$ be a group. If for some prime number~$p$ the group ring~$\mathbb{F}_p [G]$ does not have non-trivial zero divisors, then~$\mathbb{Q}[G]$ does not have non-trivial zero divisors.
\end{corollary}

The corollary justifies considering the special case where~$R = \mathbb{F}_2$, for which the conjecture is also open. Consider a group~$G$ that is a potential counterexample to the Zero Divisor Conjecture. Suppose for~$\alpha,\beta \in \mathbb{F}_2[G]\setminus \{0\}$ we have~$\alpha  \beta = 0$. Further suppose~$\alpha = g_1+\ldots+g_n$ and~$ \beta = h_1+\ldots+h_m$. W.l.o.g. we can assume that~$G$ is generated by~$\supp(\alpha) \cup \supp(\beta)$. Otherwise we replace~$G$ by the subgroup generated by this set. In order for the equation $\alpha  \beta = 0$ to hold, there must be a matching of the pairs in~$\{1,\ldots,n\} \times  \{1,\ldots,m\}$ such that for products corresponding to matched pairs equality holds, i.e., if~$(i,j)$ is matched to~$(i',j')$ then $g_{i} h_{{j}} = g_{i'} h_{{j'}}$. Figure~\ref{fig:example:of:matched:rectangle} illustrates such a matching as an example. Note that the matching is not unique if more than two of the occurring products yield the same group element. However, for our purpose it will always suffice to pick an arbitrary matching.

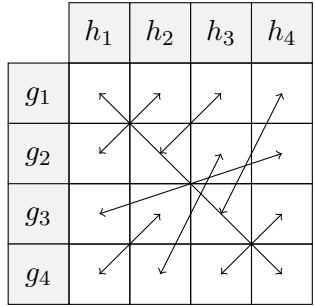
\begin{figure}[ht]
\centering
\begin{tikzpicture}[x=\celldist, y=\celldist, node distance=0 cm,outer sep = 0pt]

\glabels{4}
\hlabels{4}
\recblocks{4}{4}

\foreach \x/\y in { {(1,-2)/(2,-1)},  {(2,-2)/(3,-1)}, {(3,-3)/(4,-1)}, {(1,-4)/(2,-3)}, {(2,-4)/(3,-2)}, {(3,-4)/(4,-3)}, {(4,-2)/(1,-3)}, {(1,-1)/(4,-4)}}
{
      \draw[<->] \x -- \y {};
}
\end{tikzpicture}
\caption{A matched~$4\times 4$ rectangle.}\label{fig:example:of:matched:rectangle}
\end{figure}
Consider the finitely presented group~$\overline{G}$ given by the presentation \[\overline{G} = \langle g_1,\ldots, g_n ,h_1,\ldots, h_m \mid \{g_{i} h_{{j}} = g_{i'} h_{{j'}} \mid (i,j) \text{ is matched to } (i',j')\}\rangle.\] The group $\overline{G}$ projects homomorphically onto $G$ and by construction also contains zero divisors.
If it follows from the relations in $\overline{G}$ that two of the generators~$g_i$ are equal or two of the generators~$h_j$ are equal, then this is also the case in the group $G$. This contradicts the assumption that in~$\mathbb{F}_2 [G]$ the elements~$\alpha$ and~$\beta$ have length~$n$ and~$m$ respectively. More generally, in~$\notorsion(\overline{G})$, i.e., in the universal torsion-free image of~$\overline{G}$, all generators~$g_1,\ldots, g_n$ must be distinct and all generators~$ h_1,\ldots, h_m$  must be distinct. Recall that the \emph{universal torsion-free image} of a group is the unique quotient obtained as the limit of the process of repeatedly adding relations that force all torsion elements to be trivial. By a standard category theoretic argument, one can see that all torsion-free quotients of a group are also quotients of the universal torsion-free image (see~\cite{MR1459854}).

\begin{definition}
An~\emph{$n\times m$ matched rectangle} is a perfect matching~$M$ of the elements in~$\{(i,j) \mid i \in \{1,\ldots,n\}, j\in \{1,\ldots,m\}\}$, i.e., it is a partition of~$\{(i,j) \mid i \in \{1,\ldots,n\}, j\in \{1,\ldots,m\}\}$ into sets of size 2.

Given an~$n\times m$ matched rectangle~$M$ we define the~\emph{associated group}~$\assocgroup(M)$ as the universal torsion-free image of the group given by the presentation~\[\langle g_1,\ldots, g_n ,h_1,\ldots, h_m \mid \{g_{i} h_{j} = g_{i'} h_{{j'}} \mid (i,j) \text{ is matched to } (i',j') \text{ by } M\}\rangle.\]
\end{definition}
Note that for fixed~$n,m \in \mathbb{N}$ there are only finitely many $n\times m$ matched rectangles and consequently only finitely many associated groups.

Recall that a partial matching of a set~$S$ is a matching of a subset of~$S$. We define an $n\times m$ partially matched rectangle~$M$ to be a partial matching~$M$ of the elements in~$\{(i,j) \mid i \in \{1,\ldots,n\}, j\in \{1,\ldots,m\}\}$ and we define its associated group~$\assocgroup(M)$ in the same way in which it is defined for matched rectangles.

We say that a partially matched rectangle is \emph{degenerate} if for the associated group~$\assocgroup(M)$ it is not that case that all generators~$h_1,\ldots, h_n$ are distinct or if it is not the case that all generators~$g_1,\ldots, g_m$ are distinct. In this case we also call the associated group~$\assocgroup(M)$ degenerate.

\begin{lemma}
If~$M$ is a matched rectangle and the associated group~$\assocgroup(M)$ is not degenerate, then the group ring~$\mathbb{F}_2[\assocgroup(M)]$ contains non-trivial zero divisors.
\end{lemma}

\begin{proof}
Suppose~$M$ is a non-degenerate matched~$n\times m$ rectangle.
Consider the elements~$\alpha = g_1+\ldots+g_n$ and~$\beta = h_1+\ldots+h_m$. Since~$M$ is not degenerate, all generators in~$h_i$ are distinct. Likewise all generators~$g_j$ are distinct and thus~$\alpha, \beta \in \mathbb{F}_2[G]\setminus \{0\}$ . By definition of the associated group~$\alpha  \beta = 0$ in~$\mathbb{F}_2[G]$.
\end{proof}

We define~$\alpha_M := g_1+\ldots +g_n$ and~$\beta_M := h_1+\ldots +h_m$ as the \emph{zero divisors of~$\mathbb{F}_2[\assocgroup(M)]$ corresponding to~$M$}.

Being a universal torsion-free group, the group~$\assocgroup(M)$ is in particular torsion-free. Therefore, if~$\assocgroup(M)$ is not degenerate, it constitutes a counterexample to the Zero Divisor Conjecture over~$\mathbb{F}_2$. The lemma thus shows that a non-degenerate matched rectangle gives rise to a counterexample to the conjecture over~$\mathbb{F}_2$. The converse to this statement is also true, as shown by the following lemma.
\begin{lemma}
Let~$G$ be a torsion-free group. If~$\alpha,\beta \in \mathbb{F}_2[G]\setminus \{0\}$ with~$\alpha  \beta = 0$ then there is a non-degenerate~$\length(\alpha)\times \length(\beta)$ matched rectangle.
\end{lemma}

\begin{proof}
Suppose~$\alpha = g_1+\ldots +g_n$ and~$\beta = h_1+\ldots +h_m$. Since~$\alpha  \beta = 0$, there is a matching of~$\{1,\ldots,n\} \times  \{1,\ldots,m\}$ such that~$(i,j)$ is matched to~$(i',j')$ then $g_{i} h_{{j}} = g_{i'} h_{{j'}}$. Let~$M$ be the matched rectangle that corresponds to this matching. By definition, the associated group~$\assocgroup(M)$ maps canonically, homomorphically to the subgroup of~$G$ that is generated by~$\{g_1,\ldots,g_n,h_1,\ldots,h_m\}$. Additionally this canonical homomorphism has the property that the images of all generators~$g_i$ are distinct and the images of all generators~$h_j$ are distinct. Thus~$\assocgroup(M)$ is not degenerate.
\end{proof}

We say two partially matched rectangles are isomorphic if there is a permutation of the columns and the rows that transforms one rectangle into the other. For an~$n\times m$ matched rectangle~$M$ and an~$n\times m$ matched rectangle~$M'$ with~$n'\leq n$ and~$m'\leq m$ we say~$M'$ is a \emph{sub-rectangle} of~$M$ if there are injections~$\phi\colon \{1,\ldots,n'\}\rightarrow \{1,\ldots,n\}$ and~$\phi'\colon \{1,\ldots,m'\}\rightarrow \{1,\ldots,m\}$ such that the following holds.
If~$(i,j)$ is matched to~$(i',j')$ by $M'$ then~$(\phi(i),\phi'(j))$ is matched to~$(\phi(i'),\phi'(j'))$ by~$M$. We call the pair of maps~$(\phi,\phi')$ a \emph{homomorphism} from~$M'$ to~$M$.

\begin{lemma}\label{lem:degeneretedness:antihereditary}
Let~$M$ and~$M'$ be partially matched rectangles.
If~$M'$ is a sub-rectangle of~$M$ 
and~$\assocgroup(M')$ is degenerate, then~$\assocgroup(M)$ is degenerate.
\end{lemma}
\begin{proof}
Let~$n'\leq n$ and~$m'\leq m$ and~$(\phi,\phi')$ be a homomorphism from an~$n'\times m'$ matched rectangle~$M'$ 
to an~$n\times m$ matched rectangle~$M$. Let~$g'_1,\ldots,g'_{n'}$ and~$h'_1,\ldots,h'_{m'}$ be the generators of~$\assocgroup(M')$. 
Then the map that sends~$g_i'$ to~$\phi(g_i')$ and~$h_j'$ to~$\phi'(h_j')$ extends to a homomorphism from~$\assocgroup(M')$ to~$\assocgroup(M)$ sending generators to generators. Since~$\assocgroup(M')$ is degenerate, there are two generators~$g_j'$ and~$g'_{j'}$ which are equal or two generators~$h_j'$ and~$h'_{j'}$ which are equal. Thus~$\phi(g_i')$ and~$\phi(g'_{i'})$ are equal or~$\phi'(h'_j)$ and~$\phi'(h'_{j'})$ are equal, showing that~$M$ is degenerate.
\end{proof}

We now define a family of particularly simple partially matched rectangles that we will employ frequently.

\begin{definition}
For~$i\geq 2$ let~$M_i^{\cyc}$ be the~$2\times i$ partially matched rectangle that is given by the matching that matches~$(j,1)$ with~$(j+1,2)$ for~$j\in \{1,\ldots,i-1\}$. (See Figure~\ref{fig:cyc:rectangles}.)
\end{definition}

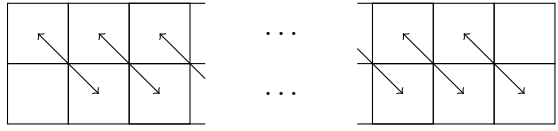
\begin{figure}[ht]
\centering
\begin{tikzpicture}[x=\celldist, y=\celldist, node distance=0 cm,outer sep = 0pt]
\recblocks{3}{2}
\node[cell]  at (7,-1) {};
\node[cell]  at (7,-2) {};
\node[cell]  at (8,-1) {};
\node[cell]  at (8,-2) {};
\node[cell]  at (9,-1) {};
\node[cell]  at (9,-2) {};

\draw[-] (2.5,-0.5) -- (3.75,-0.5) {};
\draw[-] (2.5,-1.5) -- (3.75,-1.5) {};
\draw[-] (2.5,-2.5) -- (3.75,-2.5) {};

\draw[-] (6.25,-0.5) -- (7.5,-0.5) {};
\draw[-] (6.25,-1.5) -- (7.5,-1.5) {};
\draw[-] (6.25,-2.5) -- (7.5,-2.5) {};

\foreach \x/\y in { {(1,-1)/(2,-2)},  {(2,-1)/(3,-2)}, {(8,-1)/(9,-2)}, {(7,-1)/(8,-2)}}
{
      \draw[<->] \x -- \y {};
}
\draw[<-] (3,-1) -- (3.75,-1.75) {};
\draw[->] (6.25,-1.25) -- (7,-2) {};

\node at (5,-1) {$\ldots$};
\node at (5,-2) {$\ldots$};
\end{tikzpicture}
\caption{The partially matched rectangles~$M_i^{\cyc}$.}\label{fig:cyc:rectangles}
\end{figure}

It is easy to see that if~$M_i^{\cyc}$ is a sub-rectangle of an~$m\times i$ partially matched rectangle~$M$ with~$m\geq 2$ then~$M$ is degenerate. We will prove a generalization of this statement at the end of this section.
Recall that for an element~$\alpha = g_1+\ldots +g_k$ of length~$k$, the \emph{support of~$\alpha$}, denoted~$\supp(\alpha)$, is the set~$\{g_1,\ldots,g_k\}$.

\begin{lemma}\label{lem:generators:of:rectangle}
If~$M$ is a matched rectangle without matched sub-rectangle, then for every~$h\in  \supp(\beta_M)$ the set~$\{h\}\cup \supp(\alpha_M)$ is a generating set of~$\assocgroup(M)$. If additionally~$1\in\supp(\beta_M)$ then~$\supp(\alpha_M)$ is a generating set of~$\assocgroup(M)$.
\end{lemma}

\begin{proof}
By using the defining relations, it follows that the smallest subgroup that contains~$\{h\}\cup \supp(\alpha_M)$ contains all elements in~$\supp(\beta_M)$. 
\end{proof}

For an~$n\times m$ partially matched rectangle~$M$ we define for~$(i,j) \in \{1,\ldots ,n\} \times\{1,\ldots ,m\}$ the~\emph{$(i,j)$-core subgroup} as the quotient of~$\assocgroup(M)$ obtained by adding the relations~$g_i = h_j = 1$.

\begin{lemma}
Let~$M$ be an~$n\times m$ matched rectangle. All core subgroups of~$M$ are isomorphic. 
Moreover, the group~$\assocgroup(M)$ is isomorphic to the free product of the core subgroup with the free group of rank 2.
\end{lemma}

\begin{proof}
We show that the~$(1,1)$-core subgroup is isomorphic to the~$(2,1)$-core subgroup. By symmetry this shows that all core subgroups are isomorphic. We show this by using Tietze transformations.
Let~$\langle g_1,\ldots,g_n,h_1,\ldots,h_m \mid g_1,h_1,R \rangle$ be the standard presentation for the~$(1,1)$-core subgroup and
let~$\langle a_1,\ldots,a_n,h_1,\ldots,h_m \mid a_2,h_1,R'\rangle$ be the standard presentation for the~$(2,1)$-core subgroup.
With Tietze transformations, we will transform both presentations to the same common presentation.
We add to the first presentation generators~$a_1,\ldots,a_n$ and relators~$a_i = g_2^{-1} g_i$ for all~$i\in\{1,\ldots,n\}$. We then add all relations from~$R'$, which are satisfied by construction of the core subgroups.
To the second presentation we add all generators~$g_1,\ldots,g_n$ and relators~$a_i = a_1^{-1} g_i$ for all~$i\in \{1,\ldots,n\}$. We then add all relations from~$R$, which are satisfied again by construction of the core subgroups.
It suffices now to realize that~$a_1 = g_2^{-1} g_1$ (or equivalently~$a_2 = a_1^{-1} g_2$) implies~$a_1 = g_2^{-1}$ in both groups and therefore the presentations are equivalent.

To show the second statement, let~$\langle g_1,\ldots,g_n,h_1,\ldots,h_m \mid g_1,h_1,R \rangle$ be the standard presentation for the~$(1,1)$-core subgroup. Then~$\langle x,g_1,\ldots,g_n,y,h_1,\ldots,h_m \mid g_1,h_1,R \rangle$ is a presentation of the free product of the core subgroup with the free group of rank 2.
Let~$\langle a_1,\ldots,a_n,b_1,\ldots,b_m \mid R' \rangle$ be the standard presentation of~$\assocgroup(M)$. A similar Tietze transformation technique as above, namely introducing adequate generators and the relators~$a_i = x g_i$ and~$b_j = y h_j$, shows that these two presentations present isomorphic groups.
\end{proof}

\begin{definition}
Let~$M$ be an~$n\times m$ partially matched rectangle and~$p \in \{1,\ldots,n\}\times \{1,\ldots,m\}$. The~\emph{cyclic closure} of~$p$ is the unique pair of minimal subsets~$(A,B)$ with~$A\subseteq \{1,\ldots,n\}$ and~$B\subseteq \{1,\ldots,m\}$ such that the following holds:

\begin{enumerate}
\item $p \in A\times B$ and
\item if~$(i,j) \in A \times B$ and~$(i',j')$ is paired to~$(i,j)$ then~$i' \in A$ and~$j' \in B$ or~$i' \notin A$ and~$j' \notin B$.
\end{enumerate}
\end{definition}

The term cyclic closure stems from the fact that the sets~$A$ and~$B$ generate certain cyclic subgroups of the core subgroup.

\begin{lemma}\label{lem:exhaustive:means:degenerate}
Let~$M$ be an~$n\times m$ matched rectangle. Suppose~$p \in \{1,\ldots,n\}\times \{1,\ldots,m\}$ and~$(A,B)$ is the cyclic closure of~$p$. If~$A = \{1,\ldots,n\}$ or~$B = \{1,\ldots,m\}$ then~$M$ is degenerate. 
\end{lemma}

\begin{proof}
By symmetry, it suffices to show the lemma under the assumptions that the cyclic closure~$(A,B)$ of~$p$ satisfies~$A = \{1,\ldots,n\}$. 
Consider the~$p$-core subgroup. Suppose~$p$ is matched to~$p'$ and let~$(g,h)$ be the pair of generators of the standard representation of the~$p$-core subgroup corresponding to~$p'$. By construction~$g = h^{-1}$. We show that the group generated by~$g$ contains all generators~$g_1,\ldots,g_n$ and all generators that correspond to columns in~$B$. To show this by induction, it suffices to observe that if three of the elements in~$\{g_i,g_{i'},h_j,h_{j'}\}$ are in the group generated by~$g$ and~$g_i h_j = g_{i'} h_{j'}$ then all elements in~$\{g_i,g_{i'},h_j,h_{j'}\}$ are in the group generated by~$g$. 
If the cyclic closure~$(A,B)$ satisfied~$A = \{1,\ldots,n\}$, then the group generated by the rows and columns of this cyclic closure would be cyclic and its group ring would contain non-trivial zero divisors, which gives a contradiction.
\end{proof}

\section{Solvable Baumslag-Solitar groups}\label{sec:BS:groups}

Since~$\mathbb{Z}$ is orderable, it fulfills the Zero Divisor Conjecture. The only other groups for which we will need the validity of the conjecture in this paper, are the solvable Baumslag-Solitar groups. The Baumslag-Solitar group~$BS(m,n)$ is the group given by the presentation
\[\langle a,b \mid ba^mb^{-1} = a^n\rangle.\]

The Baumslag-Solitar groups~\cite{MR0142635} are HNN-extensions. Among them are the first known examples of non-Hopfian groups. A Baumslag-Solitar group is solvable if and only if~$|m| = 1$ or~$|n| = 1$.

We argue that every quotient of a solvable Baumslag-Solitar group~$BS(1,n)$ satisfies the Zero Divisor Conjecture over~$\mathbb{F}_2$. This can also be seen from the general theorem by Kropholler, Linnell, and Moody~\cite{MR964842}, which in particular shows that solvable groups satisfy the conjecture.

\begin{theorem}\label{thm:bsg:groups}
Suppose~$n\in \mathbb{Z}$ and~$G$ is a quotient group of~$BS(1,n)$. The group ring~$\mathbb{F}_2[G]$ does not have non-trivial zero divisors.
\end{theorem}

\begin{proof}
For~$n=0$ and~$n= 1$ the group and all its factors are Abelian and thus if it is torsion-free, the group ring~$\mathbb{F}_2$ cannot have zero divisors. We thus suppose~$n\notin \{0,1\}$.
Every element in~$BS(1,n)$ can be written as~$b^i a^k b^{\ell}$ with~$i,k,\ell \in \mathbb{Z}$. 
Suppose~$G$ is a proper quotient of~$BS(1,n)$. This means that there are integers~$i,k,\ell \in \mathbb{Z}$ with~$i + \ell \neq 0$ or~$k\neq 0$ such that~$b^i a^k b^l = 1$ in~$G$, but$~b^i a^k b^{\ell} \neq 1$ in~$BS(1,n)$. Thus~$a^{-k} =  b^{i+\ell}$ in~$G$ with~$k\neq 0$ or~$i+\ell \neq 0$. If~$k = 0$ or~$i+\ell = 0$ then~$a$ or~$b$ is a torsion element. Otherwise~$a^{-k} = b^{i+\ell} = b b^{i+\ell} b^{-1} = b a^{-k} b^{-1} = a^{-nk}$ which, since~$n\neq 1$, shows that~$a$ is a torsion element. Thus~$G$ is cyclic or contains torsion. In any case it fulfills the conjecture. 

It remains to consider the case in which~$G = BS(1,n)$. Suppose~$\alpha,\beta \in \mathbb{F}_2[BS(1,n)]\setminus \{0\}$ with~$\alpha  \beta = 0$. 
Suppose further~$\alpha = g_1+\ldots+g_{n'}$ and~$\beta = h_1+\ldots+h_{m'}$.
Consider the homomorphism~$\phi$ from~$BS(1,n)$ to~$\mathbb{Z}$ that sends~$b$ to~$1$ and~$a$ to~$0$. W.l.o.g. we may assume that~$\max\{\phi(g_i) \mid i \in \{1,\ldots,n'\} \} = \max\{\phi(h_j) \mid j \in \{ 1,\ldots,m'\} \} = 0$. 
Let~$A = \phi^{-1}(0) \cap \supp(\alpha)$ and~$B = \phi^{-1}(0) \cap \supp(\beta)$. 
Then~$(\sum_{g\in A} g)  (\sum_{h\in B} h) = 0$ showing that the group ring over the kernel of~$\phi$ contains zero divisors, but this kernel is~$\mathbb{Z}[1/n]$ and every finitely generated subgroup of it is isomorphic to~$\mathbb{Z}$, which gives a contradiction.
\end{proof}

\section{Combinatorial considerations}\label{sec:by:hand}

Since an~$n\times m$ matched rectangle in particular requires a matching on a set of size~$n m$, there are no~$n\times m$ matched rectangles if both~$n$ and~$m$ are odd. By Theorem~\ref{thm:length:2}, if~$n\leq 2$ or~$m\leq 2$ then every~$n\times m$ matched rectangle is degenerate. We thus turn to longer lengths, where at least one of the integers~$n$ and~$m$ is even.

\subsection[Degeneracy of 3 x m matched rectangles]{Degeneracy of $3\times m$ matched rectangles}

Let~$M$ be a~$3\times m$ matched rectangle that is not degenerate. Since~$M$ is not degenerate, no two positions that lie in the same column are matched to each other. 
Suppose two matchings in~$M$ match the same columns~$c_1$ and~$c_2$, i.e., position~$(i,c_1)$  is matched to~$(j,c_2)$ and~$(i',c_1)$  is matched to~$(j',c_2)$  with~$i,i',j,j'\in \{1,2,3\}$ and~$i\neq i'$. Then~$M$ contains~$({M_3^{\cyc}})^T$, the transpose of~${M_3^{\cyc}}$, as a sub-rectangle, and is therefore degenerate (Lemma~\ref{lem:exhaustive:means:degenerate}). This observation allows us to associate a cubic graph with every non-degenerate~$3\times m$ matched rectangle.

\begin{definition}
Let~$M$ be a non-degenerate~$3\times m$ matched rectangle. Then~$\assocgraph(M)$, the \emph{graph underlying~$M$}, is the graph on the vertex set~$\{1,\ldots,m\}$ that has an edge from~$c_1$ to~$c_2$ if there is an edge that matches a position in column~$c_1$ to a position in column~$c_2$.
\end{definition}

\begin{theorem}
Let~$M$ be a non-degenerate~$3\times m$ matched rectangle. The graph~$\assocgraph(M)$, underlying~$M$, is a simple cubic graph that does not have a triangle.
\end{theorem}
\begin{proof}
\begin{figure}
\centering
 \subfloat[A~$3\times 3$ rectangle with cyclic core subgroup.]{
\begin{tikzpicture}[x=\celldist, y=\celldist, node distance=0 cm,outer sep = 0pt]
\recblocks{3}{3}
\foreach \x/\y in { {(1,-1)/(2,-2)}, {(2,-1)/(3,-2)},{(1,-3)/(3,-1)}}
{
      \draw[<->] \x -- \y {};
}
\label{fig:3:3:subcase:1}
\end{tikzpicture}
}
\quad \quad
 \subfloat[A~$3\times 3$ rectangle with a core subgroup isomorphic to~$BS(1,1)$.]{
\begin{tikzpicture}[x=\celldist, y=\celldist, node distance=0 cm,outer sep = 0pt]
\recblocks{3}{3}
\node[top] at (0,-1) {$1$}; 
\node[top] at (0,-2) {$\overline{g}$}; \node[top] at (0,-3) {$x$}; 
\node[top] at (1,0) {$1$}; 
\node[top] at (2,0) {$g$};\node[top] at (3,0) {$h$}; 

\foreach \x/\y in { {(1,-1)/(2,-2)}, {(2,-3)/(3,-1)},{(3,-2)/(1,-3)}}
{
      \draw[<->] \x -- \y {};
}
\label{fig:3:3:subcase:2}
\end{tikzpicture}
}
\quad \quad
 \subfloat[A~$3\times 3$ rectangle with a core subgroup isomorphic to~$BS(1,-1)$.]{
\begin{tikzpicture}[x=\celldist, y=\celldist, node distance=0 cm,outer sep = 0pt]
\recblocks{3}{3}

\node[top] at (0,-1) {$1$}; 
\node[top] at (0,-2) {$\overline{g}$}; \node[top] at (0,-3) {$x$}; 
\node[top] at (1,0) {$1$}; 
\node[top] at (2,0) {$g$}; \node[top] at (3,0) {$h$};

\foreach \x/\y in { {(1,-1)/(2,-2)}, {(2,-1)/(3,-3)},{(3,-2)/(1,-3)}}
{
      \draw[<->] \x -- \y {};
}
\label{fig:3:3:subcase:3}
\end{tikzpicture}
} \caption{Several~$3\times 3$ partially matched rectangles whose associated graph is a triangle. Here and in all following figures~$\overline{g}$ denotes~$g^{-1}$.}\label{fig:proof:3:3}
\end{figure}
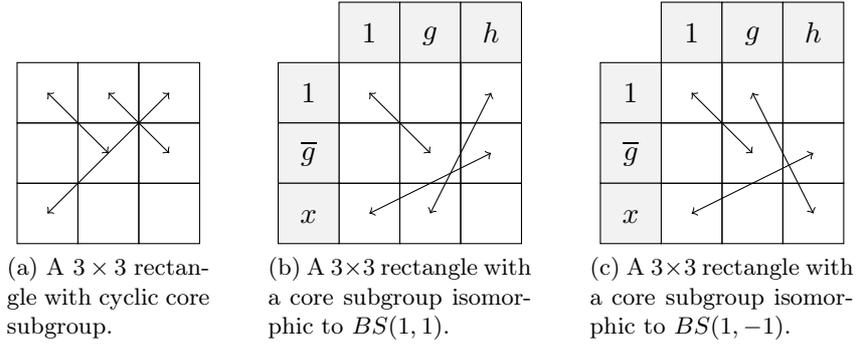
Let~$M$ be a non-degenerate~$3\times m$ matched rectangle.
By the arguments given before the definition, the graph~$\assocgraph(M)$ is simple and cubic. Suppose that the graph has a cycle of length 3.
Then~$M$ has a~$3\times 3$ sub-rectangle~$M'$ with three matching edges that form a 3-cycle. Since~$M'$ is not degenerate,~$M'$ does not contain a completely matched sub-rectangle. Up to isomorphism there are 3 possibilities for~$M'$, all depicted in Figure~\ref{fig:proof:3:3}. However,~$M'$ cannot be the rectangle shown in Figure~\ref{fig:3:3:subcase:1} by Lemma~\ref{lem:exhaustive:means:degenerate}. Suppose~$M'$ is isomorphic to the rectangle shown in Figure~\ref{fig:3:3:subcase:2}. We consider the~$(1,1)$-core subgroup of~$\assocgroup(M)$ by setting~$g_1 = 1$ and~$h_1 = 1$. We define~$g:= h_2$,~$h:= h_3$ and~$x:= g_3$. Since~$(1,1)$ is matched to~$(2,2)$, (i.e.,~$1  = g_2 \cdot h_2$) this implies~$g_2 = \overline{g}= g^{-1}$. Since~$(3,1)$ is matched to~$(2,3)$ we conclude that~$x = \overline{g} h$.
Since~$(3,2)$ is matched to~$(1,3)$ we conclude that~$xg = h$. Combining the two equations implies that~$\overline{g}hg = h$. Since~$x$ lies in the subgroup generated by~$g$ and~$h$, by Lemma~\ref{lem:generators:of:rectangle},~$g$ and~$h$ generate~$\assocgroup(M)$. Since~$\overline{g}hg = h$ is the only relation in the Baumslag-Solitar group~$BS(1,1)$, the group~$\assocgroup(M)$ is a factor group of the group~$BS(1,1)$ and thus the rectangle~$M$ is degenerate by Theorem~\ref{thm:bsg:groups}.

Suppose now that~$M'$ is isomorphic to the partially matched rectangle shown in Figure~\ref{fig:3:3:subcase:3}. By a consideration analogous to the previous case, we conclude that~$x = \overline{g} h$ and~$xh = g$. Thus~$\overline{g}h^2 = g$, which implies~$h^2 = g^2$. Again, by Lemma~\ref{lem:generators:of:rectangle},~$g$ and~$h$ generate~$\assocgroup(M)$. Recalling that the presentation~$\langle a,b \mid a^2 = b^2\rangle$ is an alternative presentation of the fundamental group of the Klein bottle, and thus a group isomorphic to~$BS(1,-1)$, we conclude that~$M$ is degenerate.
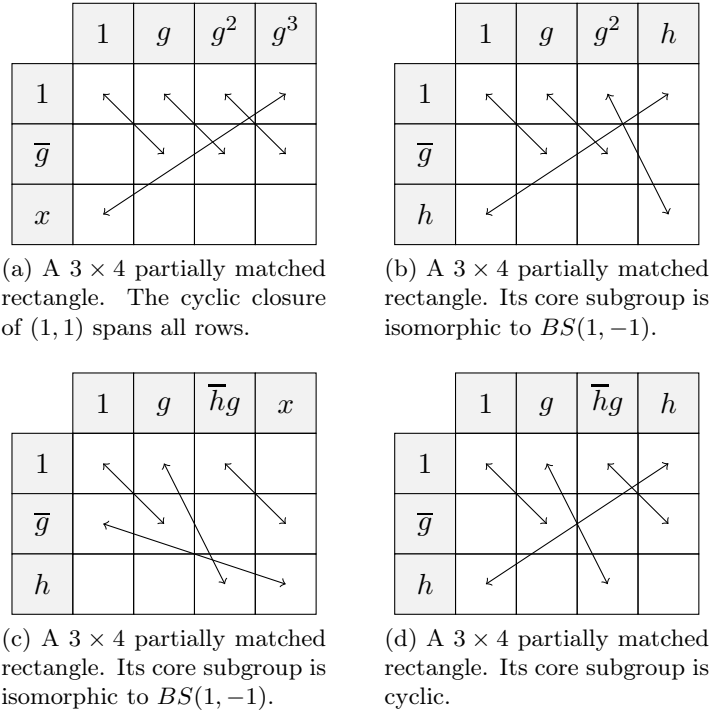
\begin{figure}
\centering
 \subfloat[A~$3\times 4$ partially matched rectangle. The cyclic closure of~$(1,1)$ spans all rows.]{
\begin{tikzpicture}[x=\celldist, y=\celldist, node distance=0 cm,outer sep = 0pt]
\recblocks{4}{3}
\node[top] at (0,-1) {$1$};  \node[top] at (0,-2) {$\overline{g}$}; \node[top] at (0,-3) {$x$}; 
\node[top] at (1,0) {$1$};  \node[top ] at (2,0) {$g$}; \node[top] at (3,0) {$g^2$}; \node[top] at (4,0) {$g^3$};
\foreach \x/\y in { {(1,-1)/(2,-2)}, {(2,-1)/(3,-2)},{(3,-1)/(4,-2)},{(4,-1)/(1,-3)}}
{
      \draw[<->] \x -- \y {};
}
\label{fig:3:4:cycle:subcase:1}
\end{tikzpicture}
}
\quad \quad
 \subfloat[A~$3\times 4$ partially matched rectangle. Its core subgroup is isomorphic to~$BS(1,-1)$.]{
\begin{tikzpicture}[x=\celldist, y=\celldist, node distance=0 cm,outer sep = 0pt]
\node[top] at (0,-1) {$1$};  \node[top] at (0,-2) {$\overline{g}$}; \node[top] at (0,-3) {$h$}; 
\node[top] at (1,0) {$1$};  \node[top ] at (2,0) {$g$}; \node[top] at (3,0) {$g^2$}; \node[top] at (4,0) {$h$};
\recblocks{4}{3}
\foreach \x/\y in { {(1,-1)/(2,-2)}, {(2,-1)/(3,-2)},{(3,-1)/(4,-3)},{(4,-1)/(1,-3)}}
{
      \draw[<->] \x -- \y {};
}
\label{fig:3:4:cycle:subcase:2}
\end{tikzpicture}
}
\\
 \subfloat[A~$3\times 4$ partially matched rectangle. Its core subgroup is isomorphic to~$BS(1,-1)$.]{
\begin{tikzpicture}[x=\celldist, y=\celldist, node distance=0 cm,outer sep = 0pt]
\node[top] at (0,-1) {$1$};  \node[top] at (0,-2) {$\overline{g}$}; \node[top] at (0,-3) {$h$}; 
\node[top] at (1,0) {$1$};  \node[top ] at (2,0) {$g$}; \node[top] at (3,0) {$\overline{h}g$}; \node[top] at (4,0) {$x$};
\recblocks{4}{3}
\foreach \x/\y in { {(1,-1)/(2,-2)},{(2,-1)/(3,-3)},{(3,-1)/(4,-2)},{(4,-3)/(1,-2)}}
{
      \draw[<->] \x -- \y {};
}
\label{fig:3:4:cycle:subcase:4}
\end{tikzpicture}
}
\quad \quad
 \subfloat[A~$3\times 4$ partially matched rectangle.  Its core subgroup is cyclic.]{
\begin{tikzpicture}[x=\celldist, y=\celldist, node distance=0 cm,outer sep = 0pt]
\node[top] at (0,-1) {$1$};  \node[top] at (0,-2) {$\overline{g}$}; \node[top] at (0,-3) {$h$}; 
\node[top] at (1,0) {$1$};  \node[top ] at (2,0) {$g$}; \node[top] at (3,0) {$\overline{h}g$}; \node[top] at (4,0) {$h$};
\recblocks{4}{3}
\foreach \x/\y in { {(1,-1)/(2,-2)},{(2,-1)/(3,-3)}, {(3,-1)/(4,-2)},{(4,-1)/(1,-3)}}
{
      \draw[<->] \x -- \y {};
}
\label{fig:3:4:cycle:subcase:5}
\end{tikzpicture}
}

 \caption{Several types of~$3\times 4$ partially matched rectangles.}\label{fig:proof:3:4:cycles}
\end{figure}
\end{proof}

Since there is no triangle-free cubic graph on 4 vertices, we obtain as corollary that every~$3\times 4$ rectangle is degenerate.
\begin{corollary}\label{cor:3:times:4}
Every~$3\times 4$ matched rectangle~$M$ is degenerate.
\end{corollary}

To show that all~$3\times 6$ matched rectangles are degenerate, we can exploit the fact that there is only one triangle free graph on 6 vertices.

\begin{lemma}\label{lem:3:times:6}
Every~$3\times 6$ matched rectangle~$M$ is degenerate.
\end{lemma}
\begin{proof}
Let~$M$ be a non-degenerate matched~$3\times 6$ rectangle. First observe that~$M$ cannot contain~$M_3^{\cyc}$: there is only one triangle free graph on 6 vertices and this graph does not have an induced path of length 4. If~$M$ were to contain~$M_3^{\cyc}$, say w.l.o.g. in the first four columns, then there must be an additional matching within the first four columns, which shows that either~$M$ has a proper matched sub-rectangle or the cyclic closure of each of the matching edges in the~$M_3^{\cyc}$ sub-rectangle spans all rows (Figure~\ref{fig:3:4:cycle:subcase:1}). Either way, this implies that~$M$ is degenerate.

We now argue that~$M$ also does not contain any of the other matched rectangles shown in Figure~\ref{fig:proof:3:4:cycles}. (In these figures, by considering the core subgroup, all elements on the left have been assigned names as depicted. For the elements on the top, the given product representations then follow from the relations determined by the rectangle.) Suppose~$M$ contained the partially matched rectangle from Figure~\ref{fig:3:4:cycle:subcase:2}. By considering the core subgroup, as indicated in Figure~\ref{fig:3:4:cycle:subcase:2}, we see that~$g^2 = h^{-2}$ and thus~$M$ would be a factor group of~$BS(1,-1)$. If~$M$ contained the partially matched rectangle shown in~Figure~\ref{fig:3:4:cycle:subcase:4} then the relations~$g^{-1}x =h^{-1}g$ and~$g^{-1} = hx$ imply~$hgh^{-1} = g^{-2}$ and we get a similar conclusion using the group~$BS(1,-2)$. Finally, in the case depicted in Figure~\ref{fig:3:4:cycle:subcase:5} we conclude from~$h^{-1}g = g^{-1}h$ that~$(h^{-1}g)^{2}= 1$ and, using torsion-freeness, obtain that~ $g = h$. This implies that in all cases~$M$ is degenerate.

We conclude the proof by showing that~$M$ must contain one of the partially matched rectangles shown in Figure~\ref{fig:proof:3:4:cycles}. By permuting rows and columns, we can achieve that position~$(1,1)$ is matched to position~$(2,2)$, and position~$(2,1)$ is matched to position~$(3,3)$: indeed,~$M$ must contain such a configuration since otherwise it contains a proper sub-rectangle or~$M_3^{\cyc}$ and is thus degenerate. We can also achieve that position~$(3,1)$ is matched to some position in column~4. Since the unique triangle free graph on 6 vertices does not contain induced paths of length~3, some position in column~4 must be matched to some position in column~1. If position~$(3,1)$ were matched to~$(4,2)$ then~$M$ would contain one of the sub-rectangles shown in~Figures~\ref{fig:3:4:cycle:subcase:4} and~\ref{fig:3:4:cycle:subcase:5}. Thus~$(3,1)$ is matched to~$(4,3)$. But then position~$(4,1)$ is not matched to column~1, since this would create one of the forbidden sub-rectangles. W.l.o.g.~$(4,1)$ is matched to column~5. It thus must be matched to~$(5,2)$ since~$M$ contained~$M_3^{\cyc}$ otherwise. By applying the induced path argument again we conclude that~$(5,1)$ must be matched to~$(2,3)$, which creates~$M_3^{\cyc}$.
\end{proof}

\subsection[Degeneracy of 4 x 4 matched rectangles]{Degeneracy of $4\times 4$ matched rectangles}

In the rest of this section we show that every~$4\times 4$ matched rectangle is degenerate. To do so, we first show that a non-degenerate~$3\times 3$ partially matched rectangle can have at most 3 matching edges.

\begin{lemma}\label{lem:3times3:no:4:matchings}
Every~$3\times 3$ partially matched rectangle with four matched positions is degenerate.
\end{lemma}
\begin{proof}
\begin{figure}[t]
\centering
\begin{tikzpicture}[x=\celldist, y=\celldist, node distance=0 cm,outer sep = 0pt]
\node[top] at (0,-1) {$1$};  \node[top] at (0,-2) {$\overline{g}$}; \node[top] at (0,-3) {$x$}; 
\node[top] at (1,0) {$1$};  \node[top] at (2,0) {$g$}; \node[top] at (3,0) {$g^2$};
\recblocks{3}{3}
\foreach \x/\y in {{(1,-1)/(2,-2)},{(2,-1)/(3,-2)}}
{  \draw[<->] \x -- \y {}; }
\node at (1,-2) {$A$}; \node at (3,-1) {$B$};

\end{tikzpicture}
\caption{A~$3\times 3$ matched rectangle that contains ${M_3^{\cyc}}$.}\label{fig:part:3:3:case}

\end{figure}

Suppose~$M$ is a non-degenerate partially matched~$3\times 3$ rectangle with four matched positions.
Since there are four matched edges but only three pairs of rows, two matchings must pair positions that lie in the same rows. Thus~$M$ contains~${M_3^{\cyc}}$ (see Figure~\ref{fig:part:3:3:case}). Suppose~$M$ is not degenerate. Positions~$A$ and~$B$ cannot be matched to each other and no positions in the third row can be matched to each other. Thus~$A$ and~$B$ are both matched to some position in the third row. Recall that by considering the core subgroup, we can assume that the generators corresponding to row 1 and column 1 respectively are the trivial group element. Since~$A$ must be matched to some position in row 3, we conclude that~$\overline{g} = g^i   x$, with~$i\in \{1,2\}$ thus~$x = g^{-k}$ with~$k\in \{2,3\}$. Since~$B$ is matched to row~3 we conclude that~$x  g^j = g^2$ with~$j\in \{0,1\}$. From~$g^{-k} g^j = g^2$ and~$j-k <2$ it follows that~$g$ is a torsion element. 
\end{proof}

\begin{lemma}\label{lem:4:times:4}
Every~$4\times 4$ matched rectangle~$M$ is degenerate.
\end{lemma}

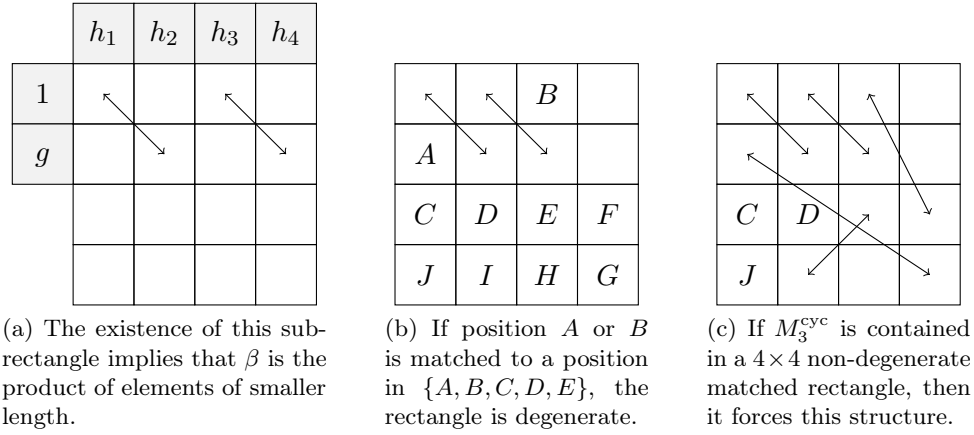
\begin{figure}[th]
\centering
 \subfloat[The existence of this sub-rectangle implies that~$\beta$ is the product of elements of smaller length.]{
\begin{tikzpicture}[x=\celldist, y=\celldist, node distance=0 cm,outer sep = 0pt]
\recblocks{4}{4}
\glabels{4}
\node[top] at (0,-1) {$1$};
\node[top] at (0,-2) {$g$};

\foreach \x/\y in { {(1,-1)/(2,-2)}, {(3,-1)/(4,-2)}}
{
      \draw[<->] \x -- \y {};
}
\label{fig:4:4:case:3}
\end{tikzpicture}
}\quad \quad
 \subfloat[If position~$A$ or~$B$ is matched to a position in~$\{A,B,C,D,E\}$, the rectangle is degenerate.]{
\begin{tikzpicture}[x=\celldist, y=\celldist, node distance=0 cm,outer sep = 0pt]
\recblocks{4}{4}
\foreach \x/\y in { {(1,-1)/(2,-2)},  {(2,-1)/(3,-2)}}
{
      \draw[<->] \x -- \y {};
}
\node at (1,-2) {$A$};
\node at (3,-1) {$B$};
\node at (4,-3) {$F$};
\node at (4,-4) {$G$};
\node at (3,-3) {$E$};
\node at (1,-4) {$J$};
\node at (2,-4) {$I$};
\node at (3,-4) {$H$};
\node at (2,-3) {$D$};
\node at (1,-3) {$C$};
\end{tikzpicture}
\label{fig:4:4:case:1}
}\qquad 
  \subfloat[If~$M_3^{\cyc}$ is contained in a~$4\times 4$ non-degenerate matched rectangle, then it forces this structure.]{
 \begin{tikzpicture}[x=\celldist, y=\celldist, node distance=0 cm,outer sep = 0pt]
 \recblocks{4}{4}
 \foreach \x/\y in { {(1,-1)/(2,-2)},  {(2,-1)/(3,-2)}, {(3,-1)/(4,-3)},{(1,-2)/(4,-4)},{(3,-3)/(2,-4)}}
 {
       \draw[<->] \x -- \y {};
 }
 
 \node at (1,-4) {$J$};
 \node at (2,-3) {$D$};
 \node at (1,-3) {$C$};
 \node at (4,-1) {};

 \end{tikzpicture}
 \label{fig:4:4:case:2}
 } %
\caption{Several types of~$4\times 4$ matched rectangles.}\label{fig:proof:4:4}
\end{figure}

\begin{proof}

Let~$M$ be a non-degenerate~$4\times 4$ matched rectangle.
Suppose first~$M$ contains the partially matched rectangle shown in Figure~\ref{fig:4:4:case:3}. Looking at the core subgroup, we see that~$\beta = h_1+h_2+h_2+h_3 = (h_1+h_3) (1+g^{-1})$. Thus~$\alpha (h_1+h_3) (1+g^{-1}) = 0$, which we conclude to be false by applying Theorem~\ref{thm:length:2} twice. By symmetry~$M$ cannot contain the transpose of the partially matched rectangle shown in Figure~\ref{fig:4:4:case:3} either.

 In~$M$ there are 8 matched pairs of positions. Since there are only~$\binom{4}{2}= 6$ ways to choose two rows, by the pigeonhole principle, there are two matching edges that match positions in the same row. We conclude that~$M$ contains~${M_3^{\cyc}}$.
 
By symmetry, we may assume that~$M$ contains the partially matched rectangle depicted in Figure~\ref{fig:4:4:case:1}. 
If~$A$ is matched to~$B$, then~$M$ is degenerate.
By Lemma~\ref{fig:part:3:3:case} we know that not both positions~$A$ and~$B$ are matched to a position in~$\{C,D,E\}$. If one of the positions~$A$ and~$B$ is matched to a position in~$\{C,D,E\}$, then at least one of the two other positions in~$\{C,D,E\}$ is not matched to~$G$, which implies that the cyclic closure of~$(1,1)$ spans all columns or all rows and thus that~$M$ is degenerate (Lemma~\ref{lem:exhaustive:means:degenerate}). By swapping rows 3 and 4 and repeating the argument, we conclude that~$A$ and~$B$ are matched to a position in~$\{F,G\}$. Thus w.l.o.g., we may assume that~$A$ is matched to~$F$ and~$B$ is matched to~$G$. 
Under this assumption, position~$E$ must be matched to~$J$ or to~$I$.
If~$E$ were matched to~$J$, then by trying every matching for~$H$ we conclude, using the observation from Figure~\ref{fig:4:4:case:1} and the fact that~$M$ does not contain proper sub-rectangles, that~$H$ cannot be matched to any other position. Thus~$E$ is matched to~$I$ but then, as shown in Figure~\ref{fig:4:4:case:2},~$M$ contains the transpose of the rectangle depicted in Figure~\ref{fig:4:4:case:3} and is thus degenerate.
\end{proof}

We have now shown all the parts of Theorem~\ref{thm:by:hand}.

\begin{proof}[Proof of Theorem~\ref{thm:by:hand}]
If~$\alpha \beta = 0$, as in the assumption of Theorem~\ref{thm:by:hand}, then by Theorem~\ref{thm:length:2} one of the lengths of~$\alpha$ and~$\beta$ must be at least 3.
Lemma~\ref{lem:Q:to:FP} shows that it is sufficient to prove the statement for~$\mathbb{F}_2$ instead of~$\mathbb{Q}$. Corollary~\ref{cor:3:times:4} and Lemmas~\ref{lem:3:times:6} and~\ref{lem:4:times:4} show exactly this.
\end{proof}

The techniques used in the proof of Theorem~\ref{thm:by:hand} foreshadow that generalizing the method to greater length combinations will result in an excessively extensive case distinction. Therefore a reasonable course of action is to proceed with a computer-assisted approach.

\section{Canonical labeling of matched rectangles}\label{sec:canonical:labeling}

Our intention is to design now an efficient algorithm for the problem of enumerating non-degenerate matched rectangles. To this end, we define first the notion of a canonical labeling for these objects. The concept of a canonical labeling can be defined in a very general context, which makes it applicable to various combinatorial objects. We refer the reader to~\cite{McKay:userguide} and~\cite{canonicalMcKay} to see how canonical labelings are used in practice for isomorph-free exhaustive generation of graphs. In this paper we only require canonical labelings for matched rectangles without proper sub-rectangles and thus state the definition directly for these. Before we do this, we first define the concept of a canonical form.

\begin{definition}
Let~$\mathcal{M}$ be the set of matched rectangles that do not contain proper sub-rectangles.
A \emph{canonical form} of the rectangles in~$\mathcal{M}$ is a map~$C  \colon \mathcal{M} \rightarrow \mathcal{M}$ such that for all rectangles~$M,M'\in \mathcal{M}$ we have 
\begin{itemize}
\item $C(M)\cong M$, 
\item $C(M) = C(M')$ if and only if~$M\cong M'$.
\end{itemize}
\end{definition}

Thus, a canonical form assigns to every rectangle~$M$ in~$\mathcal{M}$ a canonical representative~$C(M)$ isomorphic to~$M$ such that two rectangles 
have the same representative if and only if they are isomorphic.

Intuitively, a canonical labeling is a map that describes how to manipulate a matched rectangle~$M$ to obtain the canonical form~$C(M)$. 
To make this statement more precise we need some definitions. Given an~$n\times m$ matched rectangle~$M$, an~$n\times n$ permutation matrix~$S_r$ and an~$m\times m$ permutation matrix~$S_c$, we define~$S_r M S_c$ to be the matched rectangle obtained in the usual way by permuting the rows and columns according to~$S_r$ and~$S_c$ respectively. 

\begin{definition}
A \emph{canonical labeling} of the rectangles in~$\mathcal{M}$ is a map that assigns every matched rectangle~$M$ two permutation matrices~$S_r(M)$ and~$S_c(M)$ such that the map~$C \colon \mathcal{M} \rightarrow \mathcal{M}$ given by~$C(M) := S_r(M) M S_c(M)$ is a canonical form.
\end{definition}

With a different viewpoint, one can see that a canonical labeling is a map that labels the rows and columns of a matched rectangle in a way that is consistent across isomorphic matched rectangles.

In the rest of this section we will prove that for matched rectangles without proper sub-rectangles a canonical labeling can be computed in time~$\BigO(n^2m^2)$. Before we give a rigorous algorithm, we describe the intuition. Suppose we guess the row~$i$ and the column~$j$, which respectively correspond to the first row and column of the canonical isomorph of a matched rectangle~$M$. Position~$(i,j)$ is matched to some other position~$(i',j')$. The corresponding row~$i'$ and the corresponding column~$j'$ will be column 2 and row 2 of the canonical form. Now we do the following repeatedly. Among the rows and columns to which we have already assigned a number, we choose some canonical position, which has a matched partner in a row or in a column without assigned number. If the partner gives us a new row, this row is assigned the smallest unused number among the rows. Similarly, if we obtain a new column, that column is assigned the smallest unused number among the columns. 

To describe the procedure in greater detail, we need to fix some notation. For a list~$L$ we denote by~$L[a]$ entry number~$a$. By~$|L|$ we denote the number of elements in~$L$. If a list of length~$k$ contains all integers in~$\{1,\ldots,k\}$ it can also be seen as a permutation. Abusing notation, we identify such lists and their corresponding permutation matrices.

For a position~$(i,j)$ in a matched rectangle~$M$ we define~$\match_M(i,j)$ to be the position matched to~$(i,j)$. 
We require some linear order on the set of all matched rectangles in~$\mathcal{M}$ that can be computed quickly. We can, for example, use the lexicographic ordering~$<_{\lex}$, which defines~$M <_{\lex} M'$ to hold if~\[(\match_M(1,1),\match_M(1,2),\ldots,\match_M(1,m),\match_M(2,1),\ldots,\match_M(n,m))\] is lexicographically smaller than~\[(\match_{M'}(1,1),\match_{M'}(1,2),\ldots,\match_{M'}(1,m),\match_{M'}(2,1),\ldots,\match_{M'}(n,m))\] in the usual sense. However, we could also use any other efficiently computable linear order.

\paragraph{Description of Algorithm~\ref{algo}:} Algorithm~\ref{algo} receives as input a matched rectangle~$M$. It first initializes two lists~$S_r$ and~$S_c$ to be empty lists. During the execution of the algorithm, the algorithm stores in these lists the row and column permutations that give the relabeling of~$M$ that is smallest with respect to~$<_{\lex}$ among all permutations generated by the algorithm so far. For every position~$(i,j)$ the algorithm does the following. It initializes the lists~$L_r$ and~$L_c$ to~$(i)$ and~$(j)$ respectively. While~$L_r$ does not contain all rows or~$L_c$ does not contain all columns, the algorithm finds a position which is matched to a position~$(k,\ell)$ in a row not in~$L_r$ or a column not in~$L_c$. Here, with respect to some ordering induced by~$L_r$ and~$L_c$, the position that is chosen is the least position with such a matched partner. These matched partners are maintained in the ordered list~$N$. The algorithm appends, if not already contained,~$k$ and~$\ell$ to the end of~$L_r$ and~$L_c$ respectively. This procedure creates two permutations~$L_c$ and~$L_r$, which are stored in~$S_r$ and~$S_c$, whenever they give a permuted rectangle that is smaller with respect to~$<_{\lex}$ than all previous rectangles the algorithm has produced.

\begin{algorithm}[t]
\caption{Canonical labeling of matched rectangles without proper sub-rectangle}
\label{algo}
\begin{algorithmic}[1]
\REQUIRE A matched rectangle~$M$.
\ENSURE  A row permutation~$S_r$ and a column permutation~$S_c$ such that~$S_r M S_c$ is the canonical representative of~$M$.
\ENSUREGAP

\STATE $S_r \leftarrow ()$
\STATE $S_c \leftarrow ()$
\FORALL{$(i,j)\in \{1,\ldots,n\} \times \{1,\ldots,m\}$}
\STATE  $L_r \leftarrow (i)$ 
\STATE  $L_c \leftarrow (j)$ 
\STATE  $N \leftarrow ((i,j))$ 
		\WHILE{$L_r$ does not have length~$n$ or $L_c$ does not have length~$m$}
			\STATE $(k,\ell) \leftarrow \match( N[1])$ \label{access:N:1}
			\STATE pop~$N[1]$ from the list~$N$
			\IF {$k\notin L_r$}
			    \STATE append~$k$ to the end of~$L_r$
			    \STATE append~$(k,L_c[1]),(k,L_c[2]),\ldots, (k,L_c[|L_c|])$ to the end of~$N$ 
			\ENDIF
			\IF {$\ell\notin L_c$}
				\STATE append~$\ell$ to the end of~$L_c$
			    \STATE append~$(L_r[1],\ell),(L_r[2],\ell),\ldots, (L_r[|L_r|],\ell)$ to the end of~$N$ 
			\ENDIF

		\ENDWHILE
		\IF {$S_r = S_c = ()$ or $L_r M L_c <_{\lex}  S_r M S_c $}
			\STATE $S_r \leftarrow L_r$ 
			\STATE $S_c \leftarrow L_c$
		\ENDIF
\ENDFOR
\RETURN $(S_r,S_c)$
\end{algorithmic}
\end{algorithm}

\begin{theorem}\label{thm:canonical:for:recs}
Algorithm~\ref{algo} computes a canonical labeling of the rectangles in~$\mathcal{M}$ and runs, for an~$n\times m$ rectangle, in time~$\BigO(n^2m^2)$.
\end{theorem}

\begin{proof}

\emph{(Correctness).} We first prove that the output of Algorithm~\ref{algo} is indeed a canonical labeling. To show this we need to prove the following.
Let~$M$ and~$M'$ be two isomorphic matched rectangles and let~$(S^M_r,S^M_c)$ and~$(S^{M'}_r,S^{M'}_c)$ be the respective outputs of Algorithm~\ref{algo}, then~$ S^{M}_r(M) M S^{M}_c(M) =  S^{M'}_r(M') M' S^{M'}_c(M')$. To show this, it suffices to observe that the generation of the lists~$L_r$ and~$L_c$ is isomorphism invariant.
More precisely, let~$(\phi,\phi')$ be an isomorphism from~$M$ to~$M'$. For any list~$L$ let~$\phi(L)$ and~$\phi'(L)$ be the list obtained by applying~$\phi$ and~$\phi'$ to all entries of~$L$ respectively. If we compare the execution of the algorithm on rectangle~$M$ during the for-loop of some tuple~$(i,j)$ with
the execution of the algorithm on rectangle~$M$ during the for-loop of the tuple~$(\phi(i),\phi'(j))$, then any list~$L_r$ produced by the former will lead to the list~$\phi(L_r)$ produced by the latter. Likewise any list~$L_c$ produced by the former will lead to the list~$\phi(L_c)$ produced by the latter. The reason for this is that~$N$ is isomorphism invariant as well since it is determined by~$L_r$ and~$L_c$. 
For a tuple of integers~$(i,j)$, let~$L_r^M(i,j)$ and~$L_c^M(i,j)$ be the lists computed in the iteration of the for-loop that corresponds to the tuple~$(i,j)$ when~$M$ is given to the algorithm. We have thus shown for all~$(i,j)\in \{1,\ldots,n\}\times \{1,\ldots,m\}$  that~$L^M_r(i,j) M L^M_c(i,j) = L^{M'}_r(\phi(i),\phi'(j)) M' L^{M'}_c(\phi(i),\phi'(j))$. Thus~$ S^M_r(M) M S^M_c(M) =  S^{M'}_r(M') M' S^{M'}_c(M')$.

\emph{(Running time).} We now analyze the running time of the algorithm. Since comparison with respect to~$ <_{\lex}$ can be performed in~$\BigO(nm)$, it suffices to show that each iteration of the for-loop can be computed in time~$\BigO(n m)$. To show this, we observe first that in each iteration one element is removed from the list~$N$. Furthermore, since the lists~$L_r$ and~$L_c$ are monotonically increasing, all elements added to~$N$ are distinct. Since there are only~$nm$ elements in~$ \{1,\ldots,n\}\times \{1,\ldots,m\}$, the total amortized time for manipulations performed on~$N$ is in~$\BigO(nm)$. The product~$L_r M L_c$ can be computed in time~$\BigO(nm)$ as well. The other operations within the for-loop are constant time operations, since we may
assume that inverse lookup is feasible in constant time. This can be assumed since the maximal length of the lists~$L_r$ and~$L_c$ and the maximum size of any entry is~$n$ and~$m$ respectively.
\end{proof}

Theorem~\ref{thm:canonical:for:recs} stands in contrast to the fact that no polynomial time algorithm computing a canonical labeling for graphs is known. As McKay~\cite{personalMcKay} notices, the fact that matched rectangles without proper sub-rectangles can be canonized by the method described implies that their automorphism group has size at most~$nm$. Moreover, it gives as a side-effect a list of all automorphisms. This is due to the fact that once a position~$(i,j)$ has been chosen, all other positions are determined.

The requirement that the matched rectangle given to the algorithm does not have proper matched sub-rectangles is essential. Indeed, the isomorphism problem of general matched rectangles is graph isomorphism complete. A polynomial time reduction from the graph isomorphism problem to the isomorphism problem of matched rectangles can be obtained as follows. Given a graph~$G$ with vertex set~$V= \{1,\ldots,n\}$ and edge set~$E$ we define a matched~$n\times 2 |E|$ rectangle. We use two columns~$e_1$ and~$e_2$ to encode each edge~$e$. If~$e$ has end vertices~$i$ and~$j$, then positions~$(i,e_1)$ and~$(j,e_1)$ are matched to each other and positions~$(i,e_2)$ and~$(j,e_2)$ are matched to each other. For all~$k \in\{1,\ldots,n\}\setminus \{i,j\}$ we match positions~$(k,e_1)$ and~$(k,e_2)$. Two rectangles obtained in this way are isomorphic if and only if the original graphs are isomorphic. Since the construction can be performed in polynomial time this shows the claimed graph isomorphism completeness.

It is possible to adapt our canonical labeling algorithm to work also with all partially matched rectangles of a specific kind. 
In fact we will now use the algorithm to define a class of partially matched rectangles that behaves favorably with respect to the algorithm. We say a partially matched rectangle~$M$ is~\emph{adequate} if there exists a position~$(i,j)$ such that when executing the algorithm on rectangle~$M$, the for-loop corresponding to position~$(i,j)$ will create lists~$L_r$ and~$L_c$ that contain all matched positions in~$M$ before accessing a position in Line~\ref{access:N:1} as~$N[1]$ that is not matched. The definition is tailored so that Algorithm~\ref{algo} can be applied to find a canonical labeling of adequate partially matched rectangles. Furthermore, for every non-trivial adequate matched rectangle in canonical form, there is exactly one matching edge~$e$ such that removal of~$e$ gives a non-trivial adequate matched rectangle in canonical form. This edge~$e$ is the matching edge that is accessed last when the for-loop corresponding to~$(1,1)$ is executed. We call the edge~$e$ the \emph{canonical edge}. We extend the definition of a canonical edge to all adequate partially matched rectangle~$M$ by considering the pre-image of the isomorphism to the canonical form. If~$M'$ is obtained from~$M$ by deleting the canonical edge~$e$, then we call~$M'$ the \emph{canonical parent of~$M$}.

\section{Computational results}\label{sec:by:computer}

The canonical labeling algorithm from the previous section may be used to enumerate all degenerate rectangles that do not contain proper sub-rectangles.
We use the method of the canonical construction path by McKay~\cite{canonicalMcKay}. 
On a high level, the method allows us to create all matched rectangles without proper sub-rectangles by repeatedly forming canonical extensions of adequate rectangles. 

We describe the procedure in more detail: To compute all~$n\times m$ matched rectangles without proper sub-rectangles, starting with the empty~$n\times m$  rectangle, we recursively compute for a matched rectangle~$M'$ one representative of each  isomorphism class of adequate rectangles~$M$ of which~$M'$ is the canonical parent.
As argued in the previous section, the particular canonical labeling we have defined has the property that in every adequate partially matched rectangle~$M$ there is exactly one matching edge~$e$, such that deletion of~$e$ results in a partially matched proper sub-rectangle that is the canonical parent of~$M$. In order to avoid generation of isomorphic matched rectangles from the same parent~$M'$, it suffices for us to require the following. If~$M$ has a row with a matched position such that all positions in this row are unmatched in~$M'$, then this row is the smallest row in~$M'$ without matched positions. Similarly if~$M$ has a column with a matched position such that all positions in this column are unmatched in~$M'$, then this column is the smallest row in~$M'$ without matched positions. By the general theory underlying canonical construction paths, we will only generate exactly one rectangle from each isomorphism class of adequate partially matched rectangles. The canonical labeling defined for adequate rectangles has the property that the canonical parent of an adequate rectangle in canonical form is also in canonical form. Due to this property, the algorithm is in particular an orderly algorithm (see~\cite{canonicalMcKay}).

Since rectangles inherit degeneracy from their sub-rectangles (Lemma~\ref{lem:degeneretedness:antihereditary}), we can prune any partially matched rectangle for which we can determine degeneracy.
In our implementation, instead of determining degeneracy in general, we employ mainly checks for degeneracy due to the following specific reasons.

\paragraph{Obvious cyclic generation and resulting torsion:} Given a position~$p$ we can compute the cyclic closure of~$p$. All generators corresponding to rows and columns in this cyclic closure are in a cyclic subgroup. By analyzing the matchings of two positions in the cyclic closure it may be possible to find a torsion generator. This implies degeneracy. Lemma~\ref{lem:3times3:no:4:matchings} is an example of such a conclusion.

\paragraph{}We say a sequence of positions~$(i_1,j_1), (i'_1,j'_1), (i_2,j_2), (i'_2,j'_2),\ldots,(i_k,j_k), (i'_k,j'_k)$ is a matching sequence if for all~$t\in \{1,\ldots,k\}$ position~$(i_t,j_t)$ is matched to~$(i'_t,j'_t)$ and for all~$t\in \{1,\ldots,k-1\}$ we have~$j'_t = j_{t+1}$ and~$i'_t \neq  i_{t+1}$ .

\paragraph{Periodic cycle:} Suppose~$(i_1,j_1), (i'_1,j'_1),\ldots,(i_k,j_k), (i'_k,j'_k)$ is an induced matching cycle in a partially matched rectangle, i.e, the sequence is a matching sequence that additionally fulfills~$j'_k = j_1$ and for which all~$i_t$ are distinct. If there is a period~$0<p<k$ such that~$i_t = i_{t+p}$ and~$i'_t = i'_{t+p}$ for all~$t\in \{1,\ldots,k\}$ taking indices modulo~$k$, the rectangle is degenerate. Figure~\ref{fig:3:4:cycle:subcase:5} shows an example of a partially matched rectangle with this property. The period is two.

\paragraph{}Suppose for~$K\in\{1,2\}$ the sequences~$(i^K_1,j^K_1), ({i'}^K_1,{j'}^K_1),\ldots,(i^K_k,j^K_k), ({i'}^K_k,{j'}^K_k)$ are matching sequences such that~$i^1_t = i^2_t$ and~${i'_t}^1 = {i'_t}^2$ for all~$t\in\{1,\ldots,k\}$, then we say the matching sequences are parallel.

\paragraph{Mismatching parallel sequences:} 
If in a partially matched rectangle for two parallel matching sequences as above we have~${j'}^1_k = {j}^1_1$ but~${j'}^2_k \neq {j}^2_1$, then the partially matched rectangle is degenerate  since the generators corresponding to columns~${j'}^2_k$ and~${j}^2_1$ are identical.

\paragraph{}There is a situation, in which we can determine ahead of time that we will not have to consider extensions of our current partially matched rectangle, despite the partially matched rectangle not being degenerate.

\paragraph{The core subgroup is a factor group of a cyclic group or a solvable Baumslag-Solitar group:} Suppose for the partially matched rectangle~$M$, by considering the core subgroup and by using existing relations, we can determine that the core subgroup is generated by a single element. Since the Zero Divisor Conjecture holds for~$\mathbb{Z}$, any matched rectangle~$M'$ which contains~$M$ and has the same dimensions, must be degenerate. Similarly, if the canonical factor of~$M$ is a solvable Baumslag-Solitar group, then any matched rectangle~$M'$ which contains~$M$ and has the same dimensions, must be degenerate. This line of reasoning has already been applied in Lemma~\ref{lem:3:times:6}. As part of its proof, Figure~\ref{fig:3:4:cycle:subcase:1} shows an example involving a cyclic group and Figures~\ref{fig:3:4:cycle:subcase:2} and~\ref{fig:3:4:cycle:subcase:4} are examples involving Baumslag-Solitar groups.

\paragraph{}Instead of analyzing the degeneracy for all partially matched rectangles, we only apply various pruning rules. For those rectangles where the pruning rules are not sufficient to prove degeneracy, we employ the computer algebra system GAP~\cite{GAP4} due to its capability of handling finitely presented groups to eliminate them. For the different length combinations, the table in Figure~\ref{fig:exps} shows the running time of an implementation of the described algorithm and the number of unpruned rectangles that were resolved using GAP.

\begin{figure}[thb]
\begin{center}
\renewcommand{\arraystretch}{1.6}
    \begin{tabular}[t]{||c||r|r|r|r|r|r|r||}
     \hhline{|t:=:t:=t=t:=t:=t:=t:=t:=:t|}
    \multicolumn{1}{||c||}{running time} &
	\multicolumn{1}{c|}{4}
    & \multicolumn{1}{c|}{6}
    & \multicolumn{1}{c|}{8}
    & \multicolumn{1}{c|}{10}
    & \multicolumn{1}{c|}{12}
    & \multicolumn{1}{c|}{14}
    & \multicolumn{1}{c||}{16} \\
     \hhline{|:=::=======:|}
    3
    
    & $<1$s & $<1$s & $<1$s & 4s & 34s & 1877s & 111657s\\
    \hhline{||-||-------||}
    4 
    &  $<1s$  & 436s &  \multicolumn{5}{c||}{}\\
    \hhline{||-||--~~~~~||}
    
    5 
    &  17s  & \multicolumn{6}{c||}{} \\
    \hhline{||-||-~~~~~~||}
    7
    &  17570s & \multicolumn{6}{c||}{} \\
        
    \hhline{|b:=:b:======b:=:b|}
  
\multicolumn{6}{c}{}\\
     \hhline{|t:=:t:=t=t:=t:=t:=t:=t:=:t|}
    \multicolumn{1}{||c||}{unpruned rectangles} &
	   \multicolumn{1}{c|}{4}
        & \multicolumn{1}{c|}{6}
        & \multicolumn{1}{c|}{8}
        & \multicolumn{1}{c|}{10}
        & \multicolumn{1}{c|}{12}
        & \multicolumn{1}{c|}{14}
        & \multicolumn{1}{c||}{16} \\
         \hhline{|:=::=======:|}
        3
        
        & 0 & 0  & 0 & 2 & 7 & 305 & 4068\\
        \hhline{||-||-------||}
        4 
        &  3  & 585 & \multicolumn{5}{c||}{}\\
        \hhline{||-||--~~~~~||}
        
        5     &   19 & \multicolumn{6}{c||}{} \\
         \hhline{||-||-~~~~~~||}
        7 & 16715& \multicolumn{6}{c||}{} \\
        \hhline{|b:=:b:======b:=:b|}
  \end{tabular} 

  \caption[Running times and number of unpruned rectangles]{The table shows the total running times of the entire computation for various length combinations in seconds (top) and the corresponding numbers of unpruned matched rectangles, which were passed on to GAP (bottom). All computations were performed on 2.40GHz Intel Xeon E5620 cores.}
  \label{fig:exps}
  \end{center}
\end{figure}

Using the fact that the execution of the algorithm, which includes the checks performed by GAP, shows all rectangles of the various length combinations are degenerate, we can prove Theorem~\ref{thm:by:computer}.

\begin{proof}[Proof of Theorem~\ref{thm:by:computer}]
If~$\alpha \beta = 0$,  then by Theorem~\ref{thm:length:2} one of the lengths of~$\alpha$ and~$\beta$ must be at least 3.
The execution of the algorithm has shown that, for all other length combinations mentioned in the theorem, all matched rectangles are degenerate, which proves the theorem.
\end{proof}

\section{Summary and conclusion}\label{sec:conclusion}

We described a class of presentations of groups, and the corresponding combinatorial objects called matched rectangles, which are universal for the existence of a counterexample to the Zero Divisor Conjecture over~$\mathbb{F}_2$. We have designed an algorithmic method to rule out systematically counterexamples to the Zero Divisor Conjecture among products of elements of small length. The results imply the non-existence of such examples for the field~$\mathbb{Q}$. We remark that it is known that if the group ring~$R[G]$ over an integral domain~$R$ contains a non-trivial zero divisor, then it also contains a non-trivial element whose square is zero (see~\cite{Passman}). It is thus sufficient to check the conjecture only for length combinations for which~$\length(\alpha) = \length(\beta)$. However, in the construction that, given a zero divisor produces an element of square zero, it is not clear how the length changes.

Concerning the algorithmic approach, in some sense, rather than excluding certain length combinations, it is more appealing to hope that the entire procedure may produce a counterexample to the conjecture. However, even if such a counterexample in the form of a non-degenerate rectangle is produced, it is not clear how to prove then that the particular rectangle, or equivalently its associated group, is not degenerate. In particular, there is a general result by Adyan and Rabin~\cite{MR0103217,MR0110743}, which implies that given a finitely presented group, the problem of torsion-freeness is undecidable.
With respect to undecidability, Grabowski~\cite{grabowski} recently showed for a specific group, that the problem of determining whether a given element is a zero divisor is undecidable.

\section{Acknowledgments}

I am grateful for the helpful comments and the nifty suggestions I received from Brendan McKay and Matasha McConchie.

\bibliographystyle{abbrv}

\bibliography{zero_divisors}
 
\end{document}